\documentclass[reqno]{amsart}
\usepackage{amsfonts}
\usepackage{mathrsfs}
\usepackage{amsmath}
\usepackage{amsthm}
\usepackage{amssymb}
\usepackage{tikz}
\usepackage{array}
\usepackage{pifont}
\usepackage{verbatim}
\usepackage{enumitem}

\def\H{\mathcal{H}}

\def\R{\mathbb{R}}

\def\and{\qquad\text{and}\qquad}

\def\({\left(}
\def\){\right)}

\def\al{\alpha}
\def\be{\beta}

\def\cto{\hookrightarrow}

\def\td{{\rm d}}

\def\be{\begin{equation}}
\def\ee{\end{equation}}

\def\ben{\begin{eqnarray}}
\def\een{\end{eqnarray}}

\makeatletter
\@ifundefined{c@chapter}{
\newcounter{cnum}
}{
\newcounter{cnum}[chapter]
}
\def\C{\@ifnextchar[{\@with}{\@without}}
\def\@with[#1]{\@ifundefined{c@#1}{%
            \ifnum\thecnum<1
            \stepcounter{cnum}
            \fi
           \newcounter{#1}
           \setcounter{#1}{\thecnum}
            C_{\thecnum}
           \stepcounter{cnum}
        }{
        \ifnum\the\csname c@#1\endcsname<\thecnum
         C_{{\the\csname c@#1\endcsname}}
        \else
         C_{{\the\csname c@#1\endcsname}}
         \ifnum\thecnum<1
         \stepcounter{cnum}
         \fi
        \stepcounter{cnum}
        \fi
        }}
\def\@without{\ifnum\thecnum<1
                \stepcounter{cnum}
              \fi
              C_{\thecnum}
              \stepcounter{cnum}
        }
\makeatother

\newtheorem{thm}{Theorem}[section]
\newtheorem{lemma}[thm]{Lemma}

\theoremstyle{definition}
\newtheorem{Def}[thm]{Definition}

\theoremstyle{remark}

\numberwithin{equation}{section}

\begin{document}

\title[Wave equation with nonlinear damping and super-cubic nonlinearity]{Long-term behavior for wave equation with nonlinear damping and super-cubic nonlinearity}

\author{Cuncai Liu}
\address{School of Mathematics and Physics, Jiangsu University of Technology, Changzhou, 213001,China}
\email{liucc@jsut.edu.cn}
\thanks{This work was supported by the NSFC(11701230, 11801227, 11801228, 12026431, 12326365, 12442050), Natural Science Fund For Colleges and
Universities in Jiangsu Province (22KJD110001), Jiangsu 333 Project, QingLan Project of Jiangsu Province and Zhongwu Young Innovative Talents Support Program.}

\author{Fengjuan Meng}
\address{School of Mathematics and Physics, Jiangsu University of Technology, Changzhou, 213001,China}
\email{fjmeng@jsut.edu.cn}
\author{Chang Zhang}
\address{School of Mathematics and Physics, Jiangsu University of Technology, Changzhou, 213001,China}

\email{chzhnju@126.com}
\subjclass[2020]{Primary: 35B40, 35L20, 37L30.}



\keywords{Damped wave equation; Super-cubic nonlinearity; Regular solution; Strichartz estimate; Global attractor.}

\begin{abstract}
  In this paper, we consider the semilinear wave equation involving the nonlinear damping term $g(u_t) $ and nonlinearity $f(u)$. The well-posedness of the weak solution satisfying some additional regularity is achieved under the wider ranges of the exponents $g$ and $f$. Moreover, the existence of global attractor and exponential attractor are proved.
\end{abstract}

\maketitle
\setcounter{tocdepth}{1}

\numberwithin{equation}{section}

\section{Introduction}
In this paper, we are concerned with the following nonlinear damped wave equation with periodic boundary conditions
\begin{align}\label{eq1}
  \begin{cases}
    u_{tt}+g(u_{t})-\Delta u+u+f(u)=\phi, ~(t,x)\in \mathbb{R}_{+}\times{\mathbb{T}^3} , \\
    u(0,x)=u_0,~~ u_{t}(0,x)=u_1,  x\in {\mathbb{T}^3},
  \end{cases}
\end{align}
where ${\mathbb{T}^3}=[0, 2\pi]^3$ is the torus in $\mathbb{R}^3$, the external force term $\phi\in L^2({\mathbb{T}^3})$ is independent of time.

We impose the following assumptions on the nonlinear damping $g$ and the nonlinearity $f$.

\textbf{Assumptions on nonlinear damping $g$}. Let $g\in C^1(\mathbb{R})$ satisfy $g(0)=0$ and
\begin{equation}\label{G1}
  c_1(|s|^{m}-c_2)\le |g(s)| \le c_3(|s|^{m}+1),~~ 1<m\le \frac{7}{5},\tag{G1}
\end{equation}
and
\begin{equation}\label{G2}
  \gamma\le g^\prime(s)\le c(1+g(s)s)^{2/3},\tag{G2}
\end{equation}
where $\gamma>0$ is a positive constant.

\textbf{Assumptions on nonlinearity $f$}. Let $f\in C^2(\mathbb{R})$, with $f(0)=f^{\prime}(0)=0$ and satisfy the following growth condition
\begin{equation}\label{F1}
  |f^{\prime\prime}(s)|\leq c(1+|s|^{p-2}),~~ 3\le p\le 7-2m,\tag{F1}
\end{equation}
along with the dissipative condition
\begin{equation}\label{F2}
  \liminf\limits_{|s|\rightarrow\infty}\frac{f(s)}{s}> 0.\tag{F2}
\end{equation}

Concerning the phase-spaces for our problem, we consider $\H=\H^1({\mathbb{T}^3})\times L^2({\mathbb{T}^3})$ and the fractional space $\H^\sigma=H^{\sigma+1}({\mathbb{T}^3})\times H^{\sigma}({\mathbb{T}^3})$.

The well-posedness and longtime behavior of problem \eqref{eq1} has been discussed widely. Lions and Strauss\cite{Lions65} established the existence of the weak solution in $\R^3$ for the case of $2\le p\le m$. Raugel\cite{rau92} established the global attractor subject to restriction $m_1\le g^\prime(s)\le c(1+|s|^{2/3})$ with $m_1$ sufficiently large. Feireisl\cite{feireisl95_nldamp} discussed the well-posedness of problem \eqref{eq1} when $p<6m/(m+1)$, $1\le m<5$, then proved the existence of the global attractor. Chueshov and Lasiecka\cite{chueshov04} considered the existence of global attractors and their properties (structure, dimension, etc.) in the case of $m<5$, $p<3$, with large damping parameter when $p=3$, they gave a general abstract framework for discussing the asymptotic behaviors of solutions for evolutionary equation with second-order time term, see \cite{chueshov08} for more details. Sun\cite{sun06} proved the existence of the global attractor of the weak solution in the case of $r=1$, $p\le 3$, $m<5$. Under the condition $p<3$, $r=m\le5$, Nakao\cite{nakao06} gave a polynomial absorbing rate of absorbing ball and the existence of the global attractor. Khanmamedov\cite{khanmamedov06} established the existence of the global attractor of the weak solution when $p\le 3$, $m\le 5$.

With regard to the higher growth nonlinearity, some more complicated estimates are required. Recently, Todorova\cite{todorova15} has established the $L^4(0,T;L^{12}(R^3))$ a priori space-time estimate by choosing a suitable nonlinear test function in the case of $m=\frac53$. Thereby the well-posedness of weak solution in $\R^3$ was obtained, they also proved the well-posedness of strong solution on $\R^3$ for $m\ge 2$. Inspired by Todorova\cite{todorova15}, the authors\cite{Liu23} proved a priori estimate and extended the results of both the well-posedness and global attractor upto the range $p\le \min\{3m,5\}$, $m\le 5$.

We summarize the known results through Figure 1 for clarity, the well-posedness of weak soluton, the existence and regularity of the global attracrtor are all solved in region I, II and III.
\begin{center}
  \begin{tikzpicture}
    [cube/.style={very thick,black},
      grid/.style={very thin,gray},
      axis/.style={->,black,thick}]

    \draw[axis] (0,0) -- (6,0) node[anchor=west]{$p$};
    \draw[axis] (0,0) -- (0,5.5) node[anchor=west]{$m$};
    \coordinate (A) at (1,1);

    \draw (1,0) -- (1, 0.1);
    \node[below] at (1,0){1};

    \draw (3,0) -- (3, 0.1);
    \node[below] at (3,0){3};

    \draw (5,0) -- (5, 0.1);
    \node[below] at (5,0){5};

    \draw (0,1) -- (0.1, 1);
    \node[left] at (0,1){1};

    \draw (0,5/3) -- (0.1,5/3);
    \node[left] at (0,5/3){5/3};

    \draw (0,5) -- (0.1, 5);
    \node[left] at (0,5){5};

    \draw[thick, fill=black!10] (1,5) -- (1,1) -- (3,1) -- (3,5);
    \fill[black!20, domain=3:5]  (3,5) -- (3,1) -- plot(\x, {\x/(6-\x)});

    \fill[black!30, domain=3:5]  plot(\x, {\x/(6-\x)}) -- (5,5/3)--(3,1);
    \draw[thick, fill=black!40,domain=4.21:4.95] (3,1) --(4.2, 1.4)-- plot(\x, 7/2-\x/2)--(3,1) ;
    \draw[thick] (1,1) -- (1,5);
    \draw[thick] (1,1) -- (5,1);
    \draw (4.5,1.23) -- (4.7, 0.68);

    \draw[thick, dashed] (5,1) -- (5,5/3);
    \draw[thick] (1,5) -- (5,5);

    \draw[thick] (3,1) -- (5,5/3) -- (5,5);
    \draw[thick] (3,1) -- (3,5);
    \draw[thick, domain=3:5]
    plot(\x, {\x/(6-\x)});
    \fill (5,5/3) circle (2pt);

    \node[font=\fontsize{8}{8}] at (2,2.5){I};
    \node[font=\fontsize{8}{8}] at (3.7,2.5){II};
    \node[font=\fontsize{8}{8}] at (4.7,2.5){III};
    \node[font=\fontsize{4}{4}] at (4.1,1.16){\small IV};
    \node[font=\fontsize{8}{8}] at (3.9,4){$p \le \frac{6m}{m+1}$};
    \node[font=\fontsize{8}{8}] at (4.5,2) {$p\le 3m$};

    \node[font=\fontsize{8}{8}] at (4.7,0.5) {$p\le {7-2m}$};
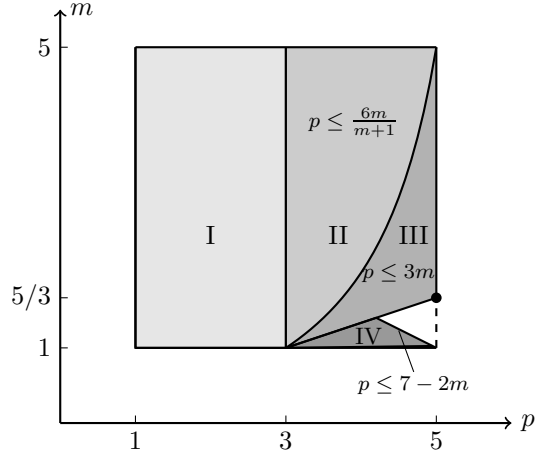
    \node at (3,-0.8) {Figure 1. Region of exponents for damping term and nonlinearity};
  \end{tikzpicture}
\end{center}

However, there still exists a gap. The a priori space-time estimate obtained in \cite{Liu23} is strongly depends on the lower growth condition of nonliear damping $g$ . However, when the growth exponent $m$ is near to $1$, the a priori space-time estimate will not be enough to deal with the higher growth of nonlinearity $f$. In this paper, we will fix the gap partially. In detail, in the case of $3\le p\le 7-2m$, we can prove the the weak solution is well-posed provided the solutoin has an extra  space-time regularity, then, the existence and regularity of the global attracrtor are established.

\medskip
Throughout the paper, the symbol $C$ stands for a generic constant indexed occasionally for the sake of clarity.

\section{Well-posedness}
\begin{Def}\emph{[Weak solution]
  A function $u$ satisfies
  $$u\in C([0,T];H^1({\mathbb{T}^3})),~~ u_t\in C([0,T];L^2({\mathbb{T}^3})),~~g(u_t)\in L^{\frac{1}{m}+1}([0,T]\times{\mathbb{T}^3})$$
  possessing the properties $u(0)=u_0$ and $u_t(0)=u_1$ is said to be a
  weak solution of problem \eqref{eq1} on $[0,T]\times{\mathbb{T}^3}$, if and only if equation \eqref{eq1} is satisfied in the sense of distribution, i.e.
  $$\int_{0}^{T}\int_{\mathbb{T}^3} \left[-u_t\psi_t+g(u_t)\psi+\nabla u\nabla \psi+u\psi+f(u)\psi\right]\td x \td t=\int_{0}^{T}\int_{\mathbb{T}^3} \phi\psi\td x\td t$$
  for any $\psi \in C_0^{\infty}((0,T)\times {\mathbb{T}^3})$.}
\end{Def}
When dealing with the Cauchy problem for nonlinear wave equations, the mixed norms $L^q_tL^r_x$ estimates called ``Strichartz estimates'' are particularly useful. The Strichartz estimates are well established on flat Euclidean space, see for example Keel and Tao \cite{keel98}, and references therein. For a compact manifold without boundary, as mentioned in \cite{sogge09}, finite speed of propagation shows that it suffices to work in coordinate charts, and to establish local Strichartz estimates for variable coefficient wave operators on Euclidean space.

\medskip
Regarding the solution $u$ of linear wave equation
\begin{eqnarray}\label{eq3.1}
  \begin{cases}
    u_{tt}-\Delta u=F(t), \text{in }\R_{+}\times{\mathbb{T}^3}, \\[3mm]
    u(0)=u_0,u_t(0)=u_1,
  \end{cases}
\end{eqnarray}
The Strichartz estimate can be stated as follows: in the following lemma, $(r',s')$ are the H\"older dual exponents to $(r,s)$.
\begin{lemma}[Strichartz estimate] \label{le2.1}
  Supposed $2<q\le \infty$, $2\le r<\infty$ and $(q,r,\sigma)$ is a triple satisfying
  \begin{equation}\label{eq2.2}
    \frac1q+\frac{3}{r}=\frac{3}{2}-\sigma \text{\quad and \quad} \frac2q+\frac{2}{r}\le 1
  \end{equation}
  and $(\tilde{q}',\tilde{r}',1-\sigma)$ satisfies the same conditions as $(q,r,\sigma)$. Then we have the following estimates for the solutions $u$ to \eqref{eq3.1} satisfying periodic boundary conditions,
  \begin{equation}\label{2.3}
    \|u\|_{L^q([-T,T];L^r({\mathbb{T}^3}))}\le C_T(\|u_0\|_{H^{\sigma}}+\|u_1\|_{H^{\sigma-1}}+\|F\|_{L^{\tilde{q}}([-T,T];L^{\tilde{r}}({\mathbb{T}^3}))}),
  \end{equation}
  where $C_T$ may depend on $T$.
\end{lemma}
In the natural phase space $\H$, i.e. in the case of $\sigma=1$, the only admissible pair for $(\tilde{q}, \tilde{r})$ is $\tilde{q}=1, \tilde{r}=2$. Since the nonlinear damping term $g(u_t)$ only belongs to $L^{\frac{m+1}{m}}([0,T]\times{\mathbb{T}^3})$, but not $L^2([0,T]\times{\mathbb{T}^3})$, we can not apply Strichartz estimates in the natural phase space $\H$ directly. To overcome this obstruction, we will apply Strichartz estimates in a lower regular phase space $\H^{\sigma-1}$ for some $\sigma<1$ such that $\tilde{r}=\frac{m+1}{m}$. In details, in the rest of the paper we will set parameters in Lemma \ref{le2.1} as follows:
\begin{equation}\label{index}
  \sigma=\frac{2}{m+1},  r=\frac{2(m+1)}{m-1}, q=m+1, \tilde{r}=\frac{m+1}{m} \text{~and~}  \tilde{q}=\frac{2(m+1)}{m+3}.
\end{equation}
It is easy to check that $(q,r,\sigma)$ and $(\tilde{q}',\tilde{r}',1-\sigma)$ both satisfy condition \eqref{eq2.2}.
\begin{Def}[Regular solution]
  \emph{A weak solution $u$ of problem \eqref{eq1} is said to be a regular solution if the following additional regularity holds: $u\in L^{q}(0,T;L^{r}({\mathbb{T}^3}))$.}
\end{Def}
\begin{thm}[Well-posedness of regular solution in region IV]\label{weak_4}
  Let $\phi\in L^2({\mathbb{T}^3})$ and conditions \eqref{G1}-\eqref{F2} hold. Then, for every $(u_0,u_1) \in \H$, there exists a regular solution $u(\cdot)$ of equation \eqref{eq1} and the following energy estimate holds
  \begin{equation}\label{energy_estimate2}
    \| (u,u_t)\|_{\H}+\| u\|_{L^{q}(t,t+1;L^{r}({\mathbb{T}^3}))}\le Q(\| (u_0,u_1)\|_{\H}+\|\phi\|),
  \end{equation}
  for any $t\ge 0$, where $Q$ is a continuous increasing function. Furthermore, the regular solution depends continuously on the initial data.
\end{thm}
\begin{proof}
  To prove the existence, we give below the formal derivation of the a priori estimate, which can be justified via Galerkin approximation. Multiplying \eqref{eq1} by $u_t(t)$ and integrating over ${\mathbb{T}^3}$, we get the energy estimate
  \begin{equation}\label{energy_est}
    E(u(t))+\int_0^t\int_{\mathbb{T}^3} g(u_t) u_t\td x\td t\le E(u(0)),
  \end{equation}
  where the energy functional
  $$E(u(t))=\frac12\|u_t\|^2+\frac12\|\nabla u\|^2+\frac12\|u\|^2+\int_{\mathbb{T}^3} F(u)\td x-\int_{\mathbb{T}^3} \phi u\td x.$$
  For any fixed $S\ge 0$, let us consider linear equation with initial data
  \begin{equation}\label{eqw}
    w_{tt}-\Delta w+w=0, \quad w(S)=u(S), w_t(S)=u_t(S).
  \end{equation}
  Then $v=u-w$ satisfies the following vanishing initial data problem
  \begin{equation}
    v_{tt}-\Delta v+v=-f(v+w)-g(u_t), \quad  v(S)=0,v_t(S)=0.
  \end{equation}
  Applying the Strichartz estimate in $\H^{\sigma-1}$ as parameters setting \eqref{index} to above equation, for any $T<1$, we have
  \begin{equation}\label{v}
    \| v\|_{L^{q}(S,S+T;L^{r}({\mathbb{T}^3}))}\le C\|f(v+w)\|_{L^{\tilde{q}}([S,S+T];L^{\tilde{r}}({\mathbb{T}^3}))}+C\|g(u_t)\|_{L^{\tilde{q}}([S,S+T];L^{\tilde{r}}({\mathbb{T}^3}))}.
  \end{equation}
  The second term will be dominated as follows
  \begin{equation}\label{gut}
    \begin{aligned}
      \|g(u_t)\|_{L^{\tilde{q}}([S,S+T];L^{\tilde{r}}({\mathbb{T}^3}))} & =\left(\int_S^{S+T} \|g(u_t)\|_{L^{(m+1)/m}({\mathbb{T}^3})}^{\tilde{q}}\text{d}t\right)^{1/\tilde{q}} \\[2mm]
                                                                      & \le CT^{\frac{3-m}{2+2m}}\|u_t\|^m_{L^{m+1}([S,S+T]\times{\mathbb{T}^3})}+CT^{\frac{m+3}{2(m+1)}}.
    \end{aligned}
  \end{equation}
  According to condition \eqref{F1}, for the first term on the right side in inequality \eqref{v}, we have
  \begin{equation}\label{vw}
    \begin{aligned}
          & \|f(v+w)\|_{L^{\tilde{q}}([S,S+T];L^{\tilde{r}}({\mathbb{T}^3}))}                                                                                \\[3mm]
      \le & C \|u^2+|v|^p+|w|^p\|_{L^{\tilde{q}}([S,S+T];L^{\tilde{r}}({\mathbb{T}^3}))}                                                                     \\[3mm]
      \le & CT^{\frac{4+2m}{3m+3}}\|u\|_{L^{\infty}([S,S+T];L^{6}({\mathbb{T}^3}))}^2+C\|\,|v|^p+|w|^p\|_{L^{\tilde{q}}([S,S+T];L^{\tilde{r}}({\mathbb{T}^3}))}.
    \end{aligned}
  \end{equation}
  By virtue of Strichartz estimate to equation \eqref{eqw} in $\H$, we can obtain
  \begin{equation}
    \|w\|_{L^{\frac{2p(m+1)}{p(m+1)-6m}}(S,S+T;L^{p\tilde{r}}({\mathbb{T}^3}))} \le C\left(E(u(S))\right)^{1/2},
  \end{equation}
  hence
  \begin{equation}\label{wp}
    \begin{aligned}
      \|\,|w|^p\|_{L^{\tilde{q}}([S,S+T];L^{\tilde{r}}({\mathbb{T}^3}))} & \le CT^{\frac{7m+3-p(m+1)}{2(m+1)}}\|w\|^{p}_{L^{\frac{2p(m+1)}{p(m+1)-6m}}(S,S+T;L^{p\tilde{r}}({\mathbb{T}^3}))} \\[3mm]
                                                                     & \le CT^{\frac{2m}{m+1}}\left(E(u(S))\right)^{p/2}.
    \end{aligned}
  \end{equation}
  Using interpolation of Lebesgue spaces, we have
  \begin{equation}\label{vp}
    \begin{aligned}
      \|\,|v|^p\|_{L^{\tilde{q}}([S,S+T];L^{\tilde{r}}({\mathbb{T}^3}))}
       & =\|v\|^p_{L^{p\tilde{q}}([S,S+T];L^{p\tilde{r}}({\mathbb{T}^3}))} \\[3mm]
       &\le  C \|v\|^{\al p}_{L^{\al p\tilde{q} }(S,S+T;L^{r}({\mathbb{T}^3}))}\|v\|^{(1-\al)p}_{L^{\infty}(S,S+T;L^{6}({\mathbb{T}^3}))} \\[3mm]
       & \le CT^{1-\frac{\al p\tilde{q}}{q}} \|v\|^{\al p}_{L^{q}(S,S+T;L^{r}({\mathbb{T}^3}))}\|v\|^{(1-\al)p}_{L^{\infty}(S,S+T;L^{6}({\mathbb{T}^3}))}.
    \end{aligned}
  \end{equation}
  where
  $$\al =(\frac16-\frac1{p\tilde{r}})/(\frac16-\frac1r).$$ Since $p\le 7-2m<6-m$, we have
  $$\al p \tilde{q}=(p-\frac6{\tilde{r}})\tilde{q}/(1-\frac6r) =\frac{p(m+1)-6m}{(2-m)(m+3)}(m+1)<m+1=q.$$

\medskip
  Substituting \eqref{gut}-\eqref{vp} into \eqref{v} and combining with the energy inequality \eqref{energy_est}, we have inequality
  \begin{equation}\label{xa}
    X(T)\le T^aAX^c(T)+BT^{b}\le AX^{c+1}(T)+AT^a+BT^{b}, \forall T\in [0, 1],
  \end{equation}
  where $X(t)=\|v\|_{L^{q}(S,S+t;L^{r}({\mathbb{T}^3}))}$ , $a=1-\frac{\al p\tilde{q}}{q}>0$, $b=\frac{2m}{m+1}>0$, $c=\al p>0$ and the constant $A$, $B$ only depend on $E(u(0))$.

\medskip
  Let $\psi(X)=X-A X^{c+1}$ and $X_0$ be the maximum point of $\psi(X)$. Since $c>0$, then $\psi(X_0)>0$. We choose $T_0>0$ such that $A T_0^a+B T_0^{b}=\frac12\psi(X_0)$, then we assert that
  \begin{equation}\label{b}
    X(t)\le X_0, \forall t\in [0, T_0].
  \end{equation}
  If not, there exists a time $t_1 \in [0, T_0]$ such that $X(t_1)>X_0$. Then, by continuity of $X(t)$, there exists a time $t_2\in [0, t_1)$ such that $X(t_2)=X_0$.
  However,
  $$X(t_2)-A X^{c+1}(t_2)=\psi(X(t_2))=\psi(X_0)> \frac12\psi(X_0)=A T_0^a+B T_0^{b}\ge A t_2^a+B t_2^{b},$$
  which contradicts with \eqref{xa}. Therefore, we can obtain that $\|v\|_{L^{q}(S,S+T_0;L^{r}({\mathbb{T}^3}))}$ can be dominated by $E(u(0))$ uniformly on $S\ge 0$.

\medskip
  To verify the continuous dependence, for any two regular solutions $u(t)$ and $v(t)$ with finite $L^{q}(0,T;L^{r}({\mathbb{T}^3}))$ norms, let the difference $\bar{u}=u-v$. Then $\bar{u}$ satisfies
  \begin{equation}\label{eqw2}
    \bar{u}_{tt}-\Delta \bar{u}+\bar{u}+g(u_t)-g(v_t)+f(u)-f(v)=0
  \end{equation}
  By the Mean Value Theorem, we have
  \begin{equation}\label{fuv}
    \begin{aligned}
      \|f(u)-f(v)\| & =\left\|\int_0^1f^{\prime}(\tau u+(1-\tau)v)\text{d}\tau \bar{u}\right\|\le C \|(1+|u|^{p-1}+|v|^{p-1}) \bar{u}\|                                                  \\[3mm]
                    & \le C(1+\|u\|^{p-1}_{L^{3p-3}({\mathbb{T}^3})}+\|v\|^{p-1}_{L^{3p-3}({\mathbb{T}^3})})\|\bar{u}\|_{L^{6}({\mathbb{T}^3})}=I(t)\|\bar{u}\|_{L^{6}({\mathbb{T}^3})}.
    \end{aligned}
  \end{equation}
  Since $3p-3\le 18-6m< r$, then by H\"older's inequality
  \begin{equation*}
    \begin{aligned}
      \|u\|^{p-1}_{L^{3p-3}({\mathbb{T}^3})}\le \|u\|^{\beta (p-1)}_{L^{r}({\mathbb{T}^3})}\|u\|^{(1-\beta)(p-1)}_{L^{6}({\mathbb{T}^3})},
    \end{aligned}
  \end{equation*}
  where
  $$\beta=(\frac16-\frac1{3p-3})/(\frac16-\frac1r).$$
  Notice that
  $$\beta (p-1)=\frac{(p-3)r}{r-6}=\frac{(p-3)(m+1)}{4-2m}\le q,$$
  thus $I(\cdot)\in L^1(0,T)$ and $f(u)-f(v)\in L^1(0,T;L^2({\mathbb{T}^3}))$. This fact allows us to multiply equation \eqref{eqw2} with $\bar{u}_t$, which implies that
 {\small  \begin{equation*}
    \frac{\text{d}}{\text{d}t}\left(\frac12\|\bar{u}_t\|^2+\frac12\|\nabla \bar{u}\|^2+\frac12\|\bar{u}\|^2\right)+\int_{\mathbb{T}^3} (g(u_t)-g(v_t))\bar{u}_t\text{d}x+\int_{\mathbb{T}^3} (f(u)-f(v))\bar{u}_t\text{d}x=0.
  \end{equation*}}Since $g(s)$ is increasing, thus
  $$\int_{\mathbb{T}^3} (g(u_t)-g(v_t))\bar{u}_t\text{d}x\ge 0.$$
  According to inequality \eqref{fuv},
  \begin{equation}
    \left|\int_{\mathbb{T}^3} (f(u)-f(v))\bar{u}_t\text{d}x\right| \le I(t)\|\bar{u}\|_{L^{6}({\mathbb{T}^3})}\|\bar{u}_t\|\le I(t)\left(\frac12\|\bar{u}_t\|^2+\frac12\|\nabla \bar{u}\|^2+\frac12\|\bar{u}\|^2 \right).
  \end{equation}
  Hence, the continuous dependence can be deduced by applying Gronwall's inequality on the following inequality
  $$\frac{\text{d}}{\text{d}t}\left(\frac12\|\bar{u}_t\|^2+\frac12\|\nabla \bar{u}\|^2+\frac12\|\bar{u}\|^2\right) \le I(t) \left(\frac12\|\bar{u}_t\|^2+\frac12\|\nabla \bar{u}\|^2+\frac12\|\bar{u}\|^2\right),$$ provided $I(t)\in L^1(0,T)$.
\end{proof}
\section{Global attractor}
Repeating arguments of the section 4 and 5 in \cite{Liu23} word by word, we can be obtain the global attractors associated with problem \eqref{eq1}.
\begin{thm}
  Let the assumptions of Theorem \ref{weak_4} hold. Then, the solution semigroup $S(t)$ generated by the regular solution of problem \eqref{eq1} possesses a global attractor $\mathscr{A}$ in the phase space $\H$. Moreover, the global attractor $\mathscr{A}$ is bounded in $\H^1$.
\end{thm}

\section{Exponential attractor}
In this section, we will prove the existence of exponential attractor provided the nonlinearity $f$ satisfies
\begin{equation}\label{F3}
  |f''(s)|\leq c_2(1+|s|^{p-2}),\tag{F3}
\end{equation}
where $3\leq p<7-2m$.

\medskip
The main result of this section is the following theorem.
\begin{thm}\label{th4.1}
  Let $\phi\in L^2({\mathbb{T}^3})$ and conditions \eqref{G1}, \eqref{G2}, \eqref{F2} and \eqref{F3} hold. Then the solution semigroup $S(t)$
  possesses an exponential attractor $\mathscr{A}_{\text{exp}}$, which is compact in $\H$ and satisfies the following properties:

  (i) $\mathscr{A}_{\text{exp}}$ is positive invariant, i.e., $S(t)\mathscr{A}_{\text{exp}}\subset \mathscr{A}_{\text{exp}}$ for all $t\geq0$;

  (ii) $dim_F \mathscr{A}_{\text{exp}}<\infty $, i.e., $\mathscr{A}_{\text{exp}}$ has finite fractal dimension in $\mathcal{H}$;

  (iii) there exists a constant $\mu>0$ such that, for any bounded set $B\subset {\H}$,
  $$ dist_{\H}(S(t)B, \mathscr{A}_{\text{exp}})\le Q(\|B\|_{\H})e^{-\mu t}\quad \text{for all } t\ge 0,$$
  where $dist_{\H}(B,A)= \sup_{x\in B}\inf_{y\in A} \|x-y\|_{\H}$ is the Hausdorff semi-distance.
\end{thm}
There are only a few results concerned on exponential attractor of wave equation with nonliear damping, Pra\v{z}\'ak\cite{prazak02,prazak03} obtained the exponential attractor under condition $m<7/3$ and $p<3$.
\subsection{The quasi-stablity and exponential attractor}
Our construction of exponential attractor depends on the quasi-stablity of semigroup.

\begin{Def}[\cite{chueshov15}, Definition 3.4.1]
  \emph{Let $(X, S(t))$ be a dynamical system in some Banach space $X$. This system is said to be quasi-stable on a set $\mathscr{B}\subset X$ (at time $t_{*}$) if there exist (a) time
  $t_{*} > 0$, (b) a Banach space $Z$, (c) a globally Lipschitz mapping $K:\mathscr{B}\to Z$, and (d) a compact seminorm $n_{Z}(\cdot)$ on the space $Z$, such that
  \begin{equation}
    \|S(t_{*})y_1-S(t_{*})y_2\|_{X}\le  q\|y_1-y_2\|_X+n_{Z}(y_1-y_2)
  \end{equation}
  for every $y_1$, $y_2\in \mathscr{B}$ with $0\le q < 1$.}
\end{Def}

\begin{thm}[\cite{chueshov15}, Theorem 3.4.7]\label{th4.3}
  Assume that a dynamical system $(X, S(t))$ is dissipative and quasi-stable on some bounded absorbing set $\mathscr{B}$ at some moment $t_*> 0$. We also assume that
  \begin{equation}
    \|S(t_{*})y_1-S(t_{*})y_2\|_{X}\le  C_{\mathscr{B}}\|y_1-y_2\|_X\text{~~for every~~} y_1, y_2\in \mathscr{B} \text{~~and~~} t\in [0, t_*]
  \end{equation}
  and there exists a space $\tilde{X}\supset X$ such that $t\mapsto S(t)y$ is H\"{o}lder continuous in $\tilde{X}$ for every $y\in \mathscr{B}$ in the sense that there exist $0<\theta\le 1$ and $C_{\mathscr{B},t_*}> 0$ such that
  \begin{equation}
    \|S(t_1)y-S(t_2)y\|_{\tilde{X}}\leq C_{\mathscr{B},t_*}|t_1-t_2|^\theta, ~~t_1, t_2\in [0,t_*], y\in \mathscr{B}.
  \end{equation}
  Then the dynamical system $(X, S(t))$ possesses a (generalized) fractal exponential attractor whose dimension is finite in the space $\tilde{X}$.
\end{thm}

\subsection{Quasi-stablity}
\begin{thm}[Quasi-stablity]\label{qs}
  Let all assumptions of Theorem \ref{th4.1} hold. Then for any regular solutions $u(t)$ and $v(t)$ of equation \eqref{eq1}, the following quasi-stable inequality holds
  \begin{equation}
    \begin{aligned}
      \|\bar{u}_t(t)\|+\|\nabla \bar{u}(t)\|+\|\bar{u}(t)\|\le & Ce^{-\mu t}(\|\bar{u}_t(0)\|+\|\nabla \bar{u}(0)\|+\|\bar{u}(0)\|) \\
                                                               & +C \sup_{s\in[0,t]}\|\bar{u}(s)\|_{H^{1-\delta}({\mathbb{T}^3})}.
    \end{aligned}
  \end{equation}
  for any $t\ge 0$, where $\bar{u}=u-v$, $\mu>0$, $\delta=\min\{(7-2m-p)/2,1/8\}$ and the constants $C$ depends on $\mathscr{B}$ olny.
\end{thm}
\begin{proof}
  The difference $\bar{u}$ satisfies equation
  \begin{equation}\label{eqbar2}
    \bar{u}_{tt}+g(u_t)-g(v_t)-\Delta \bar{u}+\bar{u}+f(u)-f(v)=0.
  \end{equation}
  Multiplying equation \eqref{eqbar2} with $\bar{u}_t+\al \bar{u}$ and integrating on ${\mathbb{T}^3}$, we can obatin
  \begin{equation}\label{eq4.6}
    \begin{aligned}
       & \frac{\text{d}}{\text{d}t}E_\al(\bar{u}(t))+\al E_\al(\bar{u}(t))+\frac\al2\|\nabla \bar{u}\|^2+\frac\al2\|\bar{u}\|^2-\al^2\int_{\mathbb{T}^3}\bar{u}_t\bar{u}\td x-\frac{3\al}{2}\|\bar{u}_t\|^2 \\[3mm]
       & +\int_{\mathbb{T}^3} (g(u_t)-g(v_t))(\bar{u}_t+\al\bar{u})\text{d}x+\int_{\mathbb{T}^3} (f(u)-f(v))(\bar{u}_t+\al\bar{u})\text{d}x=0,
    \end{aligned}
  \end{equation}
  where the modified energy functional
  $$E_\al(\bar{u}(t))=\frac12\|\bar{u}_t\|^2+\frac12\|\nabla \bar{u}\|^2+\frac12\|\bar{u}\|^2+\al\int_{\mathbb{T}^3}\bar{u}_t\bar{u}\td x.$$
  For $0<\al<1/2$, by Cauchy's inequality, we have the equivalence
  $$\frac14\|\bar{u}_t\|^2+\frac14\|\nabla \bar{u}\|^2+\frac14\|\bar{u}\|^2\le E_\al(\bar{u}(t))\le \|\bar{u}_t\|^2+\|\nabla \bar{u}\|^2+\|\bar{u}\|^2.$$
  To deal with the last but one term in \eqref{eq4.6}, let us denote
  $$g_1(u_t,v_t) =\int_0^1 g^\prime(su_t+(1-s)v_t)\td s ~~\text{and~~} g_2(u_t,v_t)=g(u_t)u_t+g(v_t)v_t,$$
  from \eqref{G2}, we have
  $$0<\gamma \le g_1(u_t,v_t) \le C[1+g_2(u_t,v_t)]^{2/3},$$
  and consequently
  \begin{equation}\label{eq4.7}
    \begin{aligned}
       & \int_{\mathbb{T}^3} \left[g(u_t(t))-g(v_t(t))\right]\bar{u}_t(t) \td x +\al\int_{\mathbb{T}^3} \left[g(u_t(t))-g(v_t(t))\right] \bar{u}(t) \td x \\[2mm]
       & =\int_{\mathbb{T}^3} g_1(u_t,v_t)|\bar{u}_t(t)|^2 \td x + \al \int_{\mathbb{T}^3} g_1(u_t,v_t) \bar{u}_t(t) \bar{u}(t) \td x                     \\[2mm]
       & \ge \frac12\int_{\mathbb{T}^3} g_1(u_t,v_t)|\bar{u}_t(t)|^2 \td x -\frac{\al^2}{2}\int_{\mathbb{T}^3} g_1(u_t,v_t) |\bar{u}(t)|^2 \td x          \\[2mm]
       & \geq\frac{\gamma}{2}\|\bar{u}_t(t)\|^2-\al^2C\left(1+\int_{\mathbb{T}^3} g_2(u_t,v_t)\td x\right) \|\bar{u}(t)\|_{L^6(\mathbb{T}^3)}^2.
    \end{aligned}
  \end{equation}
  For the last term in \eqref{eq4.6}, using condition \eqref{F3}, we can obtain that
  \begin{equation}
    \begin{aligned}
      \left|\int_{\mathbb{T}^3} (f(u)-f(v))(\bar{u}_t+\al\bar{u})\text{d}x\right| & = \left|\int_{\mathbb{T}^3} \int_0^1f^\prime(su+(1-s)v)\td s \bar{u}(\bar{u}_t+\al\bar{u})\text{d}x\right| \\[2mm]
                                                                                  & \le C\int_{\mathbb{T}^3} (1+|u|^{p-1}+|v|^{p-1}) |\bar{u}|(|\bar{u}_t|+\al|\bar{u}|)\text{d}x.
    \end{aligned}
  \end{equation}
  Then applying H\"older's inequality with exponents $\frac{3}{1-\delta}$, $\frac{6}{1+2\delta}$ and $2$ yields
  \begin{equation}
    \begin{aligned}
          & \left|\int_{\mathbb{T}^3} (f(u)-f(v))(\bar{u}_t+\al\bar{u})\text{d}x\right|                                                                                                                                  \\[2mm]
      \le & C(1+\|u\|^{p-1}_{L^{\frac{3(p-1)}{1-\delta}}({\mathbb{T}^3})}+\|v\|^{p-1}_{L^{\frac{3(p-1)}{1-\delta}}({\mathbb{T}^3})})\|\bar{u}\|_{L^{\frac{6}{1+2\delta}}({\mathbb{T}^3})}(\|\bar{u}_t\|+\al\|\bar{u}\|).
    \end{aligned}
  \end{equation}
  Using the fractional Sobolev embedding
  \begin{equation}
    H^{1-\delta}({\mathbb{T}^3}) \subset L^{\frac{6}{1+2\delta}}({\mathbb{T}^3}),
  \end{equation}
  leads to
  \begin{equation}\label{eq4.11}
    \begin{aligned}
          & \left|\int_{\mathbb{T}^3} (f(u)-f(v))(\bar{u}_t+\al\bar{u})\text{d}x\right|                                                                                                                      \\[3mm]
      \le & C(1+\|u\|^{p-1}_{L^{\frac{3(p-1)}{1-\delta}}({\mathbb{T}^3})}+\|v\|^{p-1}_{L^{\frac{3(p-1)}{1-\delta}}({\mathbb{T}^3})})\|\bar{u}\|_{H^{1-\delta}({\mathbb{T}^3})}(\|\bar{u}_t\|+\al\|\bar{u}\|) \\[3mm]
      \le & K(t,u,v)(\frac1\al\|\bar{u}\|_{H^{1-\delta}({\mathbb{T}^3})}^2+\frac{\al^2}4\|\bar{u}_t\|^2+\frac{\al^3}{4}\|\bar{u}\|^2)                                                                        \\[3mm]
      \le & \frac1{\al^2} K(t,u,v)\|\bar{u}\|_{H^{1-\delta}({\mathbb{T}^3})}^2+ \al^2 K(t,u,v) E_\al(\bar{u}(t)),
    \end{aligned}
  \end{equation}
  where
  $$K(t,u,v)=C(1+\|u\|^{p-1}_{L^{\frac{3(p-1)}{1-\delta}}({\mathbb{T}^3})}+\|v\|^{p-1}_{L^{\frac{3(p-1)}{1-\delta}}({\mathbb{T}^3})}).$$
  Since $\frac{3(p-1)}{1-\delta}\le \frac{24(p-1)}{7} <7(3-m)< r$, the H\"older's inequality with exponents $\beta=(\frac16-\frac{1-\delta}{3p-3})/(\frac16-\frac1r)$ and $1-\beta$ gives
  \begin{equation*}
    \begin{aligned}
      \|u\|^{p-1}_{L^{\frac{3(p-1)}{1-\delta}}({\mathbb{T}^3})}\le \|u\|^{\beta (p-1)}_{L^{r}({\mathbb{T}^3})}\|u\|^{(1-\beta)(p-1)}_{L^{6}({\mathbb{T}^3})}.
    \end{aligned}
  \end{equation*}
  Note that
  $$\beta (p-1)=\frac{(p-3+2\delta)r}{r-6}\le \frac{(4-2m)r}{r-6}=m+1= q,$$
  it follows energy estimate \eqref{energy_estimate2} that $K(\cdot,u,v)\in L^1_{loc}(0,+\infty)$ and $\|K\|_{L^1(t,t+1)}$ is uniformly bounded for $t\ge 0$.

\medskip
  Taking $\al<\frac14\min\{1,\gamma, \inf_{t>0}\|K\|^{-1}_{L^1(t,t+1)}\}$ and substituting \eqref{eq4.7} and \eqref{eq4.11} into \eqref{eq4.6}, one has
  \begin{equation}
    \begin{aligned}
          & \frac{\text{d}}{\text{d}t}E_\al(\bar{u}(t))+\al\left(1-\al C\int_{\mathbb{T}^3} g_2(u_t,v_t)\td x-\al K(t,u,v)\right) E_\al(\bar{u}(t)) \\[3mm]
      \le & \frac1{\al^2} K(t,u,v)\|\bar{u}\|_{H^{1-\delta}({\mathbb{T}^3})}^2.
    \end{aligned}
  \end{equation}
  From the choice of $\al$, we have
  \begin{equation}
    \begin{aligned}
      \int_0^t\left(1-\al C\int_{\mathbb{T}^3} g_2(u_t,v_t)\td x-\al K(t,u,v)\right) & \ge t-\frac t2-\al C\int_0^t \int_{\mathbb{T}^3} g_2(u_t,v_t)\td x\td s \\[2mm]
                                                                                     & \ge \frac t 2-C.
    \end{aligned}
  \end{equation}
  Then applying Gronwall's inequality implies the desired conclusion.
\end{proof}

\subsection{The H\"{o}lder continuity}
For any $z\in \H$, in order to verify the H\"{o}lder continuity of trajectory $S(t)z$, by interpolation inequality
\begin{equation}
  \begin{aligned}
    \|S(t_1)z-S(t_2)z\|_{\H} & \le \|S(t_1)z-S(t_2)z\|^{\delta}_{\H^{\delta-1}}\|S(t_1)z-S(t_2)z\|^{1-\delta}_{\H^{\delta}}                                              \\
                             & =\left\|\int_{t_1}^{t_2}\frac{d}{dt}S(s)z ds\right\|_{\H^{\delta-1}}^{\delta}\|S(t_1)z-S(t_2)z\|^{1-\delta}_{\H^{\delta}}                 \\
                             & \le \left|\int_{t_1}^{t_2}\left\|\frac{d}{dt}S(s)z\right\|_{\H^{\delta-1}}ds\right|^{\delta}\|S(t_1)z-S(t_2)z\|^{1-\delta}_{\H^{\delta}},
  \end{aligned}
\end{equation}
it needs to prove that the asymptotic regularity of $S(t)$ from $\H$ to $\H^{\delta}$ and the boundedness of $\frac{d}{dt}S(t)$ in $\H^{\delta-1}$.
\subsubsection{Asymptotic regularity}
We decompose the solution of problem \eqref{eq1} with initial data $z_0=(u_0,u_1)$ into
$$S(t)z_0=S_v(t)z_0+S_w(t)z_0,$$
where $S_v(t)z_0=(v(t),v_t(t))$ solves the problem
\begin{equation}\label{veq}
  v_{tt}+g(u_t)-g(w_t)-\Delta v+v=0, v(0)=u_0,~v_t(0)=u_1,
\end{equation}
and the remainder $S_w(t)z_0=(w(t),w_t(t))$ satisfies
\begin{equation}\label{weq}
  w_{tt}+g( w_t)-\Delta w+w+f(u)=\phi, w(0)=0,~w_{t}(0)=0.
\end{equation}

\begin{lemma}\label{lem5.5}
  Under the assumptions of Theorem \ref{th4.1}, there exists a continuous increasing function $Q$ such that
  \begin{equation}\label{wH2}
    \|w_t(t)\|_{H^{\delta}({\mathbb{T}^3})}+\|w(t)\|_{H^{1+\delta}({\mathbb{T}^3})} \leq Q(\| (u_0,u_1)\|_{\H})
  \end{equation}
  for any $t\ge 0$.
\end{lemma}
\begin{proof}
  {\it(i) The $\H$-norm estimate and dissipativity relation.}

\medskip
  Multiplying both sides of equation \eqref{veq} by $v_t$ and integrating over $(0,t)\times{\mathbb{T}^3}$, we obtain
  \begin{equation}\label{vest}
    \begin{aligned}
      \|v_t(t)\|^2+\|\nabla v(t)\|^2+\|v(t)\|^2\le  \|u_1\|^2+ \|\nabla u_0\|^2+ \|u_0\|^2, ~~\forall t\ge 0.
    \end{aligned}
  \end{equation}
  Thus $(w,w_t)=(u-v,u_t-v_t)$ is also bounded in $\H$ uniformly for $t>0$.

\medskip
  Now multiplying both sides of \eqref{weq} by $w_t$, integrating over $(0 ,t) \times {\mathbb{T}^3}$, we obtain
  \begin{equation}\label{gw}
    \int_{0}^{t}\int_{\mathbb{T}^3} g(w_t(t))w_t(t)\td x\td t \leq \|\phi\|^2+\int_{0}^{t}\int_{\mathbb{T}^3} |f(u(t))w_{t}(t)| \td x\td t.
  \end{equation}
  Analogous to \eqref{fuv}, we have $f(u)\in L^1(0,T;L^2(\mathbb{T}^3))$ and $\|f(u)\|_{L^{1}(t,t+1;L^2(\mathbb{T}^2))}\le Q(\| (u_0,u_1)\|_{\H})$. Then taking into account \eqref{energy_estimate2}, we can obtain that
  \begin{equation}
    \begin{aligned}
      \int_{0}^{t}\int_{\mathbb{T}^3} |f(u(t))w_{t}(t)| \td x\td t \le \int_{0}^{t}\|f(u(t))\| \cdot\|w_t(t)\|\td t\le (1+t)Q(\| (u_0,u_1)\|_{\H}).
    \end{aligned}
  \end{equation}
  Substituting above inequality into \eqref{gw} yields
  \begin{equation}\label{diss}
    \begin{aligned}
      \int_{0}^{t}\int_{\mathbb{T}^3} g(w_t(t))w_t(t)\td x\td t & \le (1+t)Q(\| (u_0,u_1)\|_{\H}).
    \end{aligned}
  \end{equation}

\medskip
  {\it(ii) The $\H^{1+\delta}$-norm estimate.}

\medskip
  Denoting by $A=(-\Delta)^{1/2}$. To derive the uniform bound, we will multiply equation \eqref{eq1} by $A^{2\delta} w_t+\alpha A^{2\delta}w$. However, since $A^{2\delta} w_t$ belongs to $H^{-2\delta}$, it cannot directly serve as a test function. Therefore, a formal derivation of the a priori estimate is presented next.

\medskip
  Multiplying both sides of \eqref{weq} by $A^{2\delta} w_t+\alpha A^{2\delta}w$ with $\al \in (0, 1)$ and $\delta=\min\{(7-2m-p)/2,1/5\}$, then integrating over ${\mathbb{T}^3}$, we find
  \begin{equation}\label{Delta_w}
    \begin{aligned}
       & \frac{\td}{\td t}\Phi(t) +\al\Phi(t)+\int_{\mathbb{T}^3} g(w_t)A^{2\delta} w_t\td x+\frac\al2\|A^{1+\delta}w\|^{2}+\frac\al2\|A^{\delta}w\|^{2}-\frac{3\al}{2}\|A^{\delta}w_{t}\|^{2} \\
       & -\al^2\int_{\mathbb{T}^3}A^{\delta} w_tA^{\delta} w\td x+\al\int_{\mathbb{T}^3} g(w_t)A^{2\delta} w\td x-\int_{\mathbb{T}^3} f^\prime(u)u_tA^{2\delta} w\td x=0.
    \end{aligned}
  \end{equation}
  where
  \begin{equation*}
    \begin{aligned}
      \Phi(t)= & {\frac{1}{2}}\|A^{\delta}w_{t}\|^{2}+{\frac{1}{2}}\|A^{1+\delta}w\|^{2}+\frac12\|A^{\delta}w\|^2+\al\int_{\mathbb{T}^3}A^{\delta}w_tA^{\delta} w\td x \\
               & +\int_{\mathbb{T}^3} f(u)A^{2\delta} w\td x-\int_{\mathbb{T}^3} \phi A^{2\delta} w\td x.
    \end{aligned}
  \end{equation*}
  Since $5-2\delta\ge p+2m-2\ge p$,
  \begin{equation}\label{eq4_23}
    \begin{aligned}
      \int_{\mathbb{T}^3} f(u)A^{2\delta} w\td x & \le \|f(u)\|_{L^{\frac{6}{5-2\delta}}({\mathbb{T}^3})}\|A^{2\delta} w\|_{L^{\frac{6}{1+2\delta}}({\mathbb{T}^3})} \\[2mm]
                                                 & \le C\|f(u)\|_{L^{\frac{6}{p}}({\mathbb{T}^3})}\|A^{2\delta} w\|_{H^{1-\delta}({\mathbb{T}^3})}                   \\[2mm]
                                                 & =C\|f(u)\|_{L^{\frac{6}{p}}({\mathbb{T}^3})}(\|A^{1+\delta}w\|^2+\|A^{\delta}w\|^2)^{1/2}                         \\[2mm]
                                                 & \le C\|u\|^{2p}_{H^1({\mathbb{T}^3})}+\frac18(\|A^{1+\delta}w\|^2+\|A^{\delta}w\|^2).
    \end{aligned}
  \end{equation}
  we have
  \begin{equation}\label{eq4_24}
    \Phi(t)\ge \frac14\left(\|A^{\delta}w_{t}\|^{2}+ \|A^{1+\delta}w\|^{2}+\|A^{\delta}w\|^2\right)-C,
  \end{equation}
  we just need the lower bound estimate of $\Phi(t)$ since $\Phi(0)=0$.

\medskip
  According to \cite[Lemma 2.1]{savostianov14},
  \begin{equation}
    \int_{\mathbb{T}^3} g(w_t)A^{2\delta} w_t\td x=c\int_{\R^3}\int_{\mathbb{T}^3}\frac{(g(w_t(x+h))-g(w_t(x)))(w_t(x+h)-w_t(x))}{|h|^{3+2\delta}}\td x\td h.
  \end{equation}
  By the Mean Value Theorem, we have
  \begin{equation}
    g(w_t(x+h))-g(w_t(x))=\int_0^1 g^\prime(\tau w_t(x+h)+(1-\tau)w_t(x))\td \tau(w_t(x+h)-w_t(x)).
  \end{equation}
  Then by virtue of \eqref{G2},
  \begin{equation}
    \int_{\mathbb{T}^3} g(w_t)A^{2\delta} w_t\td x\ge c\gamma \int_{\R^3}\int_{\mathbb{T}^3}\frac{(w_t(x+h)-w_t(x))^2}{|h|^{3+2\delta}}\td x\td h=\gamma\|A^{\delta}w_{t}\|^{2}.
  \end{equation}

\medskip
  For the first term on the last line of \eqref{Delta_w}, we have
  \begin{equation}
    \begin{aligned}
      \int_{\mathbb{T}^3} g(w_t)A^{2\delta} w\td x & \le \|g(w_t)\|_{L^{\frac{2}{m}}({\mathbb{T}^3})}\|A^{2\delta} w\|_{L^{\frac{2}{2-m}}({\mathbb{T}^3})} \\[2mm]
                                                   & \le C(1+\|w_t\|^m)\|A^{2\delta} w\|_{L^{\frac{6}{1+2\delta}}({\mathbb{T}^3})}                         \\[2mm]
                                                   & \le C(1+\|w_t\|^m)\|A^{2\delta} w\|_{H^{1-\delta}({\mathbb{T}^3})}                                    \\[2mm]
                                                   & \le  Q(\| (u_0,u_1)\|_{\H})\|A^{2\delta} w\|_{H^{1-\delta}({\mathbb{T}^3})}.
    \end{aligned}
  \end{equation}
  For the second term on the last line of \eqref{Delta_w}, by H\"older's inequality with conjugate exponents $(\frac{3}{1-\delta}, 2,\frac{6}{1+2\delta})$, we have
  \begin{equation}\label{eq4_29}
    \begin{aligned}
      \int_{\mathbb{T}^3} f^\prime(u)u_tA^{2\delta} w\td x & \le \|f^\prime(u)\|_{L^{\frac3{1-\delta}}({\mathbb{T}^3})} \|u_t\|\cdot \|A^{2\delta} w\|_{L^{\frac{6}{1+2\delta}}({\mathbb{T}^3})} \\[2mm]
                                                           & \le C\|f^\prime(u)\|_{L^{\frac3{1-\delta}}({\mathbb{T}^3})} \|u_t\|\cdot \|A^{2\delta} w\|_{H^{1-\delta}({\mathbb{T}^3})}           \\[2mm]
                                                           & \le C(1+\|u\|^{p-1}_{L^{\frac{3(p-1)}{1-\delta}}({\mathbb{T}^3})}) \|u_t\|\cdot \|A^{2\delta} w\|_{H^{1-\delta}({\mathbb{T}^3})}    \\[2mm]
                                                           & =K(t)\|A^{2\delta} w\|_{H^{1-\delta}({\mathbb{T}^3})}.
    \end{aligned}
  \end{equation}
  To evaluate $K(t)$, since $\frac{3(p-1)}{1-\delta}\le \frac{15(p-1)}{4}\le \frac{15(3-m)}{2}< r$, by H\"older's inequality
  \begin{equation*}
    \begin{aligned}
      \|u\|^{p-1}_{L^{\frac{3(p-1)}{1-\delta}}({\mathbb{T}^3})}\le \|u\|^{\beta (p-1)}_{L^{r}({\mathbb{T}^3})}\|u\|^{(1-\beta)(p-1)}_{L^{6}({\mathbb{T}^3})},
    \end{aligned}
  \end{equation*}
  where
  $$\beta=(\frac16-\frac{1-\delta}{3p-3})/(\frac16-\frac1r).$$
  Notice that
  $$\beta (p-1)=\frac{(p-3+2\delta)r}{r-6}\le \frac{(4-2m)r}{r-6}=m+1= q,$$
  thus $K(\cdot)\in L^1_{loc}(0,+\infty)$ and $\|K\|_{L^1(t,t+1)}\le Q(\| (u_0,u_1)\|_{\H})$ for $t\ge 0$.

\medskip
  Choosing $\al$ be small enough, substituting inequalities \eqref{eq4_23}-\eqref{eq4_29} into \eqref{Delta_w} yields
  \begin{equation}\label{}
    \begin{aligned}
      \frac{\td}{\td t} (\Phi(t)+C)^{1/2}+ \al (\Phi(t)+C)^{1/2}\le K(t)+Q(\| (u_0,u_1)\|_{\H}).
    \end{aligned}
  \end{equation}
  We can derive from Gronwall's inequality that
  \begin{equation}\label{5.40}
    (\Phi(t)+C)^{1/2}\le C^{1/2}e^{-\al t}+\int_0^t e^{\al (s-t)}(K(s)+Q(\| (u_0,u_1)\|_{\H}))\td s\le Q(\| (u_0,u_1)\|_{\H}).
  \end{equation}
  Above arguments can be justified rigorously using the Galerkin approximation method. We know that there exists a complete orthonormal basis of $L^2(\mathbb{T}^3)$, $\{e_i\}_{i\in \mathbb{N}}$ made of eigenvectors of $-\Delta$ with Dirichlet boundary condition. Define the finite dimensional projection $P_n$ as
  $$P_nu=\sum_{i=1}^{n}a_ie_i\text{\quad for\quad} u=\sum_{i=1}^{\infty}a_ie_i\in L^2(\mathbb{T}^3).$$

\medskip
  Since $\{e_i\}$ are smooth functions, then $P_n u$ is smooth with respect to the spatial variable.

\medskip
  For the approximating equation,
  \begin{equation}\label{apeq}
  u^n_{tt}+P_n g(u^n_{t})-\Delta u^n+u^n+P_nf(u^n)=P_n\phi,  u^n(0,x)=P_n u_0,~~ u^n_{t}(0,x)=P_nu_1,
  \end{equation}
  we also decompose the solution to $u^n(t)=v^n(t)+w^n(t)$, where $v^n(t)$ solves the problem
  \begin{equation}\label{vneq}
    v^n_{tt}+P_ng(u^n_t)-g(w^n_t)-\Delta v^n+v^n=0, v^n(0)=P_nu_0,~v^n_t(0)=P_nu_1,
  \end{equation}
  and the remainder $w^n(t)$ satisfies
  \begin{equation}\label{wneq}
    w^n_{tt}+P_ng(w^n_t)-\Delta w^n+w^n+P_nf(u)=P_n\phi, w^n(0)=0,~w^n_{t}(0)=0.
  \end{equation}
  Then $v^n(t)$ and $w^n(t)$ have the same estimates as \eqref{vest} and \eqref{5.40}. Then passing the limit $n\to \infty$ and applying the lower semi-continuity of norms, we can derive the desired regularity estimate on $w(t)$.

\medskip
  This completes the proof.
\end{proof}

\begin{lemma}\label{lem5.6}
  Under the assumptions of Theorem \ref{th4.1}, there exists positive constant
  $\rho=\rho(\| (u_0,u_1)\|_{\H})>0$ such that
  \begin{equation}\label{4.31}
    \|v_t(t)\|+\|\nabla v\|+\|v\|\le C e^{-\rho t}\| (u_0,u_1)\|_{\H}.
  \end{equation}
\end{lemma}
\begin{proof}
  Multiplying both sides of equation \eqref{veq} by $v_t + \al v$, $\al \in (0, 1)$ and integrating over ${\mathbb{T}^3}$, we obtain
  \begin{equation}\label{4.32}
    \begin{aligned}
       & {\frac{\td}{\td t}}\left(\|v_t\|^2+\|\nabla v\|^2+\|v(t)\|^2+2\al \langle v_t,v\rangle\right)+2\al \|\nabla v\|^{2}+2\al \| v\|^{2}-2\al\| v_t\|^{2} \\
       & +2\int_{\mathbb{T}^3} \left[g(u_t)-g(w_t)\right](v_t+\al v)\td x=0.
    \end{aligned}
  \end{equation}
  For the last term, similar to the derivation of \eqref{eq4.7}, we have
  \begin{equation}
    2\int_{\mathbb{T}^3} (g(u_t)-g(w_t))(v_t+\al v)\text{d}x \ge \gamma\|v_t\|^2-\al^2C\left(1+\int_{\mathbb{T}^3} g_2(u_t,w_t)\td x\right) \|v(t)\|_{L^6(\mathbb{T}^3)}^2.
  \end{equation}
  So taking into account this inequality in \eqref{4.32} with $\al<\min\{\gamma/4,1/2\}$ we find
  \begin{equation}\label{5.44}
    \begin{aligned}
          & {\frac{\td}{\td t}}\left(\|v_t\|^2+\|\nabla v\|^2+\| v\|^{2}+2\al \langle v_t,v\rangle\right)+\al (\| v_t\|^{2}+\|\nabla v\|^{2}+\| v\|^{2}+2\al \langle v_t,v\rangle) \\
      \le & \al^2C \int_{\mathbb{T}^3} g_2(u_t,w_t)\td x \|v\|_{L^6(\mathbb{T}^3)}^2                                                                                               \\
      \le & \al^2C \int_{\mathbb{T}^3} g_2(u_t,w_t)\td x (\| v_t\|^{2}+\|\nabla v\|^{2}+\| v\|^{2}+2\al \langle v_t,v\rangle).
    \end{aligned}
  \end{equation}
  From  $\int_{0}^{+\infty}\int_{{\mathbb{T}^3}}g(u_{t})u_{t}\td x\td t\le Q(\| (u_0,u_1)\|_{\H})$ and \eqref{diss}, it follows that
  \begin{equation}
    \al^2C\int_{0}^{t}\int_{\mathbb{T}^3} g_2(u_t,w_t)\td x \td s \le  (1+t)\al^2CQ(\| (u_0,u_1)\|_{\H})\le \frac\al2(1+t).
  \end{equation}
  Taking $\al$ to be small enough such that
  \begin{equation}
    \al^2C\int_{0}^{t}\int_{\mathbb{T}^3} g_2(u_t,w_t)\td x \td s \le \frac\al2(1+t),
  \end{equation}
  then by Gronwall's inequality, we have
  \begin{equation}
    \begin{aligned}
      \|v_t(t)\|^2+\|\nabla v(t)\|^2+\| v(t)\|^{2} & \le C\left(\|u_1\|^2+\|\nabla u_0\|^2+\| u_0\|^{2}\right)e^{-\frac\al2 t}.
    \end{aligned}
  \end{equation}
\end{proof}

We now turn to the asymptotic regularity.
\begin{lemma}\label{le6.7}
  Under the assumptions of Theorem \ref{th4.1}, there exists a compact positive invariant set $\mathscr{B}_1$ in $\H^{\delta}$ and a positive constant $\mu$ such that, for any bounded set $B\subset {\H}$,
  $$ dist_{\H}(S(t)B, \mathscr{B}_1)\le Q(\|B\|_{\H})e^{-\mu t}\quad \text{for all } t\ge 0.$$
\end{lemma}

\begin{proof}
  Let $\mathscr{B}_0$ be an absorbing set of $S(t)$, denote $\mathscr{K}=\overline{\bigcup_{t\ge 0}S_w(t)\mathscr{B}_0}^{\H}$, recall that by Lemma \ref{lem5.5}, $\mathscr{K}$ is bounded in $\H^{\delta}$. We will show that $\mathscr{K}$ is an exponentially attracting set.

\medskip
  Since $\mathscr{B}_0$ is an absorbing set, then for any given bounded set $B\subset \H$, there exists a $T=T(\|B\|_{\H})$ such that, $S(t)B\subset \mathscr{B}_0$, for all $t\ge T$. This implies that, for $t\ge T$,
  \begin{equation}
    \begin{aligned}
      dist_{\H}(S(t)B, \mathscr{K}) & = dist_{\H}(S(t-T)S(T)B, \mathscr{K})                      \\[2mm]
                                    & \le dist_{\H}(S(t-T)S(T)B, S_w(t-T)\mathscr{B}_0)          \\[2mm]
                                    & \le dist_{\H}(S(t-T)\mathscr{B}_0,  S_w(t-T)\mathscr{B}_0) \\[2mm]
                                    & \le\|S_v(t-T)\mathscr{B}_0\|_{\H}.
    \end{aligned}
  \end{equation}
  Combining this with \eqref{4.31} leads to
  \begin{equation}
    \begin{aligned}
      dist_{\H}(S(t)B, \mathscr{K})\le\|S_v(t-T)\mathscr{B}_0\|_{\H}\le Ce^{-\rho(t-T)}\|\mathscr{B}_0\|_{\H}.
    \end{aligned}
  \end{equation}
  On the other hand, by energy estimate \eqref{energy_estimate2}, we have the uniform estimate
  \begin{equation}
    dist_{\H}(S(t)B, \mathscr{K}) \le Q(\|B\|_{\H}),\quad t \ge 0.
  \end{equation}
  Combining the above two inequalities, we obtain
  \begin{equation}
    dist_{\H}(S(t)B, \mathscr{K}) \le (Q(\|B\|_{\H})e^{\rho T}+Ce^{\rho T}\|\mathscr{B}_0\|_{\H})e^{-\rho t},\quad t \ge 0.
  \end{equation}

\medskip
  Towards to invariant, we define $\mathscr{B}_1=\overline{\bigcup_{t\ge 0}S(t)\mathscr{K}}^{\H}$, then $\mathscr{B}_1$ is a positive invariant set and also satisfies exponentially attracting property
  \begin{equation}
    dist_{\H}(S(t)B, \mathscr{B}_1) \le (Q(\|B\|_{\H})e^{\rho T}+Ce^{\rho T}\|\mathscr{B}_0\|_{\H})e^{-\rho t},\quad t \ge 0.
  \end{equation}

\medskip
  It remains to show that $\mathscr{B}_1$ is bounded in $\H^{\delta}$. Multiplying both sides of \eqref{eq1} by $A^{2\delta} u_t+\alpha A^{2\delta}u$ with $\al \in (0, 1)$, then repeating the arguments from \eqref{eq4_23} to \eqref{eq4_29} word by word, it follows that $S(t)\mathscr{K}$ is bounded in $\H^{\delta}$ uniformly for $t\ge 0$. This implies that $\mathscr{B}_1$ is bounded in $\H^{\delta}$, which completes the proof.
\end{proof}
\subsubsection{The H\"{o}lder continuity}
Let's firstly prove the $\H^{\delta-1}$ boundedness of operator $\frac{d}{dt}S(t)$ on $\R^{+}\times \mathscr{B}_1$.
\begin{lemma}\label{le4.9}
  Given $z_0=(u_0,u_1)\in \mathscr{B}_1$, then $\frac{d}{dt}S(t)z_0=(u_t,u_{tt})\in L^\infty(\R^{+};\H^{\delta-1})$ and for any $t\ge 0$,
  \begin{equation}\label{eq6.14}
    \left\|\frac{d}{dt}S(t)z_0\right\|_{\H^{\delta-1}}\le Q(\|\mathscr{B}_1\|_{\H^{\delta}}).
  \end{equation}
\end{lemma}
\begin{proof}
  According to the fact that $\mathscr{B}_1$ is positive invariant, we deduce that $S(t)z_0=(u(t),u_t(t))\in \mathscr{B}_1\subset \H^{\delta}$. Therefore, it remains to show that $u_{tt} \in H^{\delta-1}$.

\medskip
  Following from \eqref{eq1}, we have $u_{tt}=\phi-g(u_{t})+\Delta u-u-f(u)$. For the damping term, since $\delta=\min\{(7-2m-p)/2,1/8\}\le 1/8$ and $m\le 7/5$, then $6m/(5-2\delta)\le 2$, we have
  \begin{equation}
    \|g(u_t)\|_{H^{\delta-1}({\mathbb{T}^3})}\le C(1+\|u_t\|^m_{L^{\frac{6m}{5-2\delta}}(\mathbb{T}^3)}) \le C(1+\| u_t\|^{m}).
  \end{equation}
  By the embedding
  \begin{equation}
    L^{\frac{6}{5-2\delta}}(\mathbb{T}^3)=\big(L^{\frac{6}{1+2\delta}}(\mathbb{T}^3))'\cto\big(H^{1-\delta}(\mathbb{T}^3)\big)'=H^{\delta-1}(\mathbb{T}^3),
  \end{equation}
  we can obtain that
  \begin{equation}
    \|f(u)\|_{H^{\delta-1}({\mathbb{T}^3})}\le C\|f(u)\|_{L^{\frac{6}{5-2\delta}}(\mathbb{T}^3)}\le C(1+\|u\|^p_{L^{\frac{6p}{5-2\delta}}(\mathbb{T}^3)}).
  \end{equation}
  It follows form $\delta \le (5-p)/2$ that $6p/(5-2\delta)\le 6$, therefore
  \begin{equation}
    \|f(u)\|_{H^{\delta-1}({\mathbb{T}^3})}\le C(1+\|u\|^p_{L^{\frac{6p}{5-2\delta}}(\mathbb{T}^3)}) \le C(1+\| u\|^{p}_{H^1(\mathbb{T}^3)}).
  \end{equation}

\medskip
  Consequently, all the right terms in equation $u_{tt}=\phi-g(u_{t})+\Delta u-u-f(u)$ belong to $L^\infty(\R^{+};H^{\delta-1})$, thus $u_{tt}\in L^\infty(\R^{+};H^{\delta-1})$ and $\|u_{tt}\|_{L^\infty(\R^{+};H^{\delta-1})}\le Q(\|\mathscr{B}_1\|_{\H^{\delta}})$. The proof is complete.
\end{proof}
Now, we can verify the H\"{o}lder continuity.
\begin{lemma}\label{le6.10}
  The mapping $t\mapsto S(t)z$ is H\"older continuous of order $\delta$ on $[0,t_*]$ for any $t_{*}>0$ and  $z\in\mathscr{B}_1$.
\end{lemma}
\begin{proof} For $z\in \mathscr{B}_1$, and $t_1,t_2\in [0,t_{*}]$ we have
  \begin{equation}\label{eq6.16}
    \begin{aligned}
      \|S(t_1)z-S(t_2)z\|_{\H} & \le \|S(t_1)z-S(t_2)z\|^{\delta}_{\H^{\delta-1}}\|S(t_1)z-S(t_2)z\|^{1-\delta}_{\H^{\delta}}                                             \\
                               & =\left\|\int_{t_1}^{t_2}\frac{d}{dt}S(s)z ds\right\|_{\H^{\delta-1}}^{\delta}\|S(t_1)z-S(t_2)z\|^{1-\delta}_{\H^{\delta}}                \\
                               & \le \left|\int_{t_1}^{t_2}\left\|\frac{d}{dt}S(s)z\right\|_{\H^{\delta-1}}ds\right|^{\delta}\|S(t_1)z-S(t_2)z\|^{1-\delta}_{\H^{\delta}}
    \end{aligned}
  \end{equation}
  and
  \begin{equation}\label{eqq6.17}
    \|S(t_1)z-S(t_2)z\|_{\H^{\delta}}\le \|S(t_1)z\|_{\H^{\delta}}+\|S(t_2)z\|_{\H^{\delta}}\le 2\|\mathscr{B}_1\|_{\H^{\delta}}.
  \end{equation}
  According to Lemma \ref{le4.9}, we have
  \begin{equation}\label{eqq6.18}
    \left|\int_{t_1}^{t_2}\left\|\frac{d}{dt}S(s)z\right\|_{\H^{\delta-1}}ds\right|\le\left\|\frac{d}{dt}S(\cdot)z\right\|_{L^\infty(\R^{+};\H^{\delta-1})}|t_1-t_2|\le Q(\|\mathscr{B}_1\|_{\H^{\delta}})|t_1-t_2|.
  \end{equation}
  Substituting \eqref{eqq6.17}$, $ \eqref{eqq6.18} into \eqref{eq6.16}, we obtain
  \begin{equation*}
    \|S(t_1)z-S(t_2)z\|_{\H}\le Q(\|\mathscr{B}_1\|_{\H^{\delta}}) |t_1-t_2|^\delta.
  \end{equation*}
\end{proof}

\begin{proof}[\bf{Proof of Theorem \ref{th4.1}}]By Theorem \ref{weak_4}, Thoerem \ref{qs} and Lemma \ref{le4.9}, all assumptions of Theorem \ref{th4.3} are satisfied, thus the local dynamical system $(S(t),\mathscr{B}_1)$ has an exponential attractor $\mathscr{A}_{\text{exp}}$ such that
  $$dist_{\H}(S(t)\mathscr{B}_1,\mathscr{A}_{\text{exp}})\le \C e^{-\nu t}, \forall t\ge 0,$$
  for some $\nu>0$.

  Recall that $\mathscr{B}_1$ is an exponential attracting set, then apply the transitivity property of exponential attraction\cite{fab04}, we conclude that $\mathscr{A}_{\text{exp}}$ attracts $B$ exponentially.
\end{proof}


\end{document}